\documentclass{article}
\usepackage[british]{babel}
\usepackage{amsthm}
\usepackage{amssymb, enumerate}
\usepackage{latexsym}
\usepackage{url}
\usepackage[all,cmtip]{xy}
\usepackage{amsmath,amssymb}

\newcounter{itemcounter}
\numberwithin{itemcounter}{section}

\newtheorem{thm}[itemcounter]{Theorem}
\newtheorem{lem}[itemcounter]{Lemma}
\newtheorem{defi}[itemcounter]{Definition}
\newtheorem{prop}[itemcounter]{Proposition}

\newtheorem{con}[itemcounter]{Conjecture}

\newtheorem*{thm*}{Theorem}
\newtheorem*{con*}{Conjecture}
\newtheorem*{cor*}{Corollary}
\newtheorem*{ack*}{Acknowledgements}

\newcommand{\Syl}{\mathop{\rm Syl}\nolimits}
\newcommand{\Irr}{\mathop{\rm Irr}\nolimits}
\newcommand{\IBr}{\mathop{\rm IBr}\nolimits}

\newcommand{\Aut}{\mathop{\rm Aut}\nolimits}

\newcommand{\Stab}{\mathop{\rm Stab}\nolimits}

\newcommand{\rk}{\mathop{\rm rk}\nolimits}

\newcommand{\mf}{\mathop{\rm mf}\nolimits}
\newcommand{\omf}{\mathop{\rm mf_\mathcal{O}}\nolimits}

\newcommand{\nth}{\mathop{\rm th}\nolimits}

\newcommand{\Id}{\mathop{\rm Id}\nolimits}

\newcommand{\cO} {\mathcal{O}}

\title{Arbitrarily large $\cO$-Morita Frobenius numbers \footnote{This research was supported by the EPSRC (grant no. EP/T004606/1).}}
\author{Michael Livesey\footnote{School of Mathematics, University of Manchester, Manchester, M13 9PL, United Kingdom. Email: michael.livesey@manchester.ac.uk}}
\date{}

\begin{document}

\maketitle

\begin{abstract}
We construct blocks of finite groups with arbitrarily large $\cO$-Morita Frobenius numbers. There are no known examples of two blocks defined over $\cO$, with isomorphic defect groups, that are not Morita equivalent but the corresponding blocks defined over $k$ are. Therefore, the above strongly suggests that Morita Frobenius numbers are also unbounded, which would answer a question of Benson and Kessar.
\end{abstract}

\section{Introduction}

Let $l$ be a prime, $(K,\cO,k)$ an $l$-modular system with $k$ algebraically closed, $H$ a finite group and $b$ a block of $\cO H$. In this setup we always assume $K$ contains a primitive $|H|^{\nth}$ root of unity. We define $\overline{\phantom{A}}:\cO\to k$ to be the natural quotient map which we extend to the corresponding ring homomorphism $\overline{\phantom{A}}:\cO H\to kH$. For $n\in\mathbb{N}$, we define $\overline{b}^{(l^n)}$ to be the block of $kH$ that is the image of $\overline{b}$ under the following ring automorphism
\begin{align}\label{algn:ring_auto}
\begin{split}
kH&\to kH\\
\sum_{h\in H}\alpha_hh&\mapsto\sum_{h\in H}(\alpha_h)^{l^n}h.
\end{split}
\end{align}
We can also define the corresponding permutation of the blocks of $\cO H$. In other words, $b^{(l^n)}$ is the unique block of $\cO H$ such that $\overline{b^{(l^n)}}=\overline{b}^{(l^n)}$. We now define the Morita Frobenius number of a block, first defined by Kessar~\cite{ke04}.

\begin{defi}
Let $H$ be a finite group and $b$ a block of $\cO H$. The \emph{Morita Frobenius number} of $\overline{b}$, denoted by $\mf(\overline{b})$, is the smallest $n\in\mathbb{N}$, such that $\overline{b}$ and $\overline{b}^{(l^n)}$ are Morita equivalent as $k$-algebras. Similarly, the \emph{$\cO$-Morita Frobenius number} of $b$, denoted by $\omf(b)$, is the smallest $n\in\mathbb{N}$, such that $b$ and $b^{(l^n)}$ are Morita equivalent as $\cO$-algebras.
\end{defi}

Since a Morita equivalence between two blocks defined over $\cO$ implies a Morita equivalence between the corresponding blocks defined over $k$, we always have $\mf(\overline{b})\leq \omf(b)$. Donovan's conjecture, which can be stated over $\cO$ or $k$, is as follows.

\begin{con}[Donovan] Let $L$ be a finite $l$-group. Then, amongst all finite groups $H$ and blocks $b$ (respectively $\overline{b}$) of $\cO H$ (respectively $kH$) with defect groups isomorphic to $L$, there are only finitely many Morita equivalence classes.
\end{con}

A consequence of Donovan's conjecture stated over $\cO$ (respectively over $k$) is that $\cO$-Morita Frobenius numbers (respectively Morita Frobenius numbers) are bounded in terms of a function of the isomorphism class of the defect group.

\begin{con}\label{con:bdmf}
Let $L$ be a finite $l$-group. Then, amongst all finite groups $H$ and blocks $b$ (respectively $\overline{b}$) of $\cO H$ (respectively $kH$) with defect groups isomorphic to $L$, $\omf(b)$ (respectively $\mf(\overline{b})$) is bounded.
\end{con}

In~\cite[Theorem 1.4]{ke04} Kessar proved that Donovan's conjecture stated over $k$ is equivalent to Conjecture~\ref{con:bdmf} stated over $k$ together with the so-called Weak Donovan conjecture.

\begin{con}[Weak Donovan]
Let $L$ be a finite $l$-group. Then there exists $c(L)\in\mathbb{N}$ such that if $H$ is a finite group and $b$ is a block of $\cO H$ with defect groups isomorphic to $L$, then the entries of the Cartan matrix of $b$ are at most $c(L)$.
\end{con}

In~\cite[Theorem 3.11]{eaeili19} Eaton, Eisele and the author proved that Donovan's conjecture stated over $\cO$ is equivalent to Conjecture~\ref{con:bdmf} stated over $\cO$ together with the Weak Donovan conjecture.
\newline
\newline
The question of whether Morita Frobenius numbers of blocks defined over $k$ have a universal bound is one that has gained much interest in recent years. In~\cite[Examples 5.1,5.2]{bk07} Benson and Kessar constructed blocks with Morita Frobenius number two, the first discovered to be greater than one. The relevant blocks all have a normal, abelian defect group and abelian $l'$ inertial quotient with a unique isomorphism class of simple modules. It was also proved that among such blocks the Morita Frobenius numbers cannot exceed two~\cite[Remark 3.3]{bk07}. In work of Benson, Kessar and Linckelmann~\cite[Theorem 1.1]{bkl2018} the bound of two was extended to blocks that don't necessarily have a unique isomorphism class of simple modules. It was also shown that the bound of two applies to $\cO$-Morita Frobenius numbers of the corresponding blocks defined over $\cO$. Finally, Farrell~\cite[Theorem 1.1]{far17} and Farrell and Kessar~\cite[Theorem 1.1]{fake19} proved that the $\cO$-Morita Frobenius number of any block of a finite quasi-simple group is at most four.
\newline
\newline
Our main result (see Theorem~\ref{thm:main}) is as follows:

\begin{thm*}
For every prime $l$ and $n\in\mathbb{N}$, there exists an $\cO$-block $b$ with $\omf(b)=n$.
\end{thm*}

The proof of the above theorem presented in this paper relies heavily on the fact that the blocks are defined over $\cO$. Specifically, the result~\cite[Propostion 4.4]{eali19} quoted in the proof of Proposition~\ref{prop:indMor} is a result that ultimately depends on Weiss' criterion~\cite{we88}. Weiss' criterion is a result concerning permutation modules for groups algebras of $l$-groups that does not hold over $k$.
\newline
\newline
Notwithstanding the previous paragraph, there are no known examples of two blocks defined over $\cO$, with isomorphic defect groups, that are not Morita equivalent but the corresponding blocks defined over $k$ are. Therefore, Theorem~\ref{thm:main} strongly suggests that Morita Frobenius numbers are also unbounded. This would answer two questions posed by Benson and Kessar~\cite[Questions 6.2,6.3]{bk07}. Note that for a fixed $l$, the blocks constructed in Theorem~\ref{thm:main} do not have bounded defect. Therefore, the theorem does not contradict Conjecture~\ref{con:bdmf}.
\newline
\newline
The following notation will hold throughout this article. If $H$ is a finite group and $b$ a block of $\cO H$, then we set $\Irr(H)$ (respectively $\IBr(H)$) to be the set of ordinary irreducible (respectively irreducible Brauer) characters of $H$ and $\Irr(b)\subseteq\Irr(H)$ (respectively $\IBr(b)\subseteq\IBr(H)$) the set of ordinary irreducible (respectively irreducible Brauer) characters lying in the block $b$. If $N\lhd H$ and $\chi\in\Irr(N)$, then we denote by $\Irr(H,\chi)$ the set of irreducible characters of $H$ appearing as constituents of $\chi\uparrow^H$. Similarly we define $\Irr(b,\chi):=\Irr(b)\cap\Irr(H,\chi)$. $1_H\in\Irr(H)$ will designate the trivial character of $H$. We use $e_b\in\cO H$ to denote the block idempotent of $b$. Similarly if $H$ is a $p'$-group and $\varphi\in\Irr(H)$, then we use $e_\varphi\in\cO H$ to signify the block idempotent corresponding to $\varphi$. Finally we set $[h_1,h_2]:=h_1^{-1}h_2^{-1}h_1h_2$ for $h_1,h_2\in H$.
\newline
\newline
The article is organised as follows. In $\S$\ref{sec:prelim} we establish some preliminaries about Morita equivalences between blocks of finite groups. $\S$\ref{sec:def} contains the definition of the blocks $B_\varphi$ that are then used in $\S$\ref{sec:main} to prove our main theorem.

\section{Morita equivalences between blocks}\label{sec:prelim}

For any free $\cO$-module $M$, we will denote by $\rk_{\cO}(M)$ its rank as a module over $\cO$. All $\cO$-modules considered will have finite rank.

\begin{lem}\label{lem:rank}
Let $H_1$ (respectively $H_2$) be a finite group and $b_1$ (respectively $b_2$) a block of $\cO H_1$ (respectively $\cO H_2$) such that $\rk_{\cO}(b_1)=\rk_{\cO}(b_2)$. Then any $b_1$-$b_2$-bimodule $M$ inducing a Morita equivalence between $b_1$ and $b_2$ must satisfy $\rk_{\cO}(M)\leq\rk_{\cO}(b_1)$.
\end{lem}

\begin{proof}
Let $M$ be such a bimodule and $\sigma:\Irr(b_1)\to\Irr(b_2)$ the corresponding bijection of characters. Then
\begin{align*}
K\otimes_{\cO}M \cong \bigoplus_{\chi\in\Irr(b_1)}V_\chi \otimes_K V_{\sigma(\chi)}^*,
\end{align*}
where $V_\chi$ denotes a $KH_1$-module affording $\chi$ and $V_{\sigma(\chi)}^*$ denotes the dual of a $KH_2$-module affording $\sigma(\chi)$. Therefore,
\begin{align*}
\rk_{\cO}(M)&=\sum_{\chi\in\Irr(b_1)}\chi(1)\sigma(\chi)(1)
\leq\sqrt{\left(\sum_{\chi\in\Irr(b_1)}\chi(1)^2\right)\left(\sum_{\psi\in\Irr(b_2)}\psi(1)^2\right)}\\
&=\sqrt{\rk_{\cO}(b_1)\rk_{\cO}(b_2)}=\rk_{\cO}(b_1),
\end{align*}
where the inequality follows from the Cauchy-Schwarz inequality.
\end{proof}

Let $b$ be a block of $\cO H$, for some finite group $H$ and $Q$ a normal $p$-subgroup of $H$. We denote by $b^Q$ the direct sum of blocks of $\cO(H/Q)$ dominated by $b$, that is those blocks not annihilated by the image of $e_b$ under the natural $\cO$-algebra homomorphism $\cO H\to \cO(H/Q)$. Also, for any pair of finite groups $H_1,H_2$ and $\cO H_1$-$\cO H_2$-bimodule $M$ we routinely view $M$ as an $\cO(H_1\times H_2)$ via $(h_1,h_2).m=h_1mh_2^{-1}$, for $h_1\in H_1$, $h_2\in H_2$ and $m\in M$.

\begin{prop}\label{prop:indMor}
Let $H_1$ (respectively $H_2$) be a finite group, with $Q_1\lhd H_1$ (respectively $Q_2\lhd H_2$) a normal $l$-subgroup and $b_1$ (respectively $b_2$) a block of $\cO H_1$ (respectively $\cO H_2$).
\begin{enumerate}
\item Suppose $M$ is a $b_1$-$b_2$-bimodule inducing a Morita equivalence between $b_1$ and $b_2$ such that the corresponding bijection $\Irr(b_1)\to\Irr(b_2)$ restricts to a bijection $\Irr(b_1,1_{Q_1})\to\Irr(b_2,1_{Q_2})$. Then ${}^{Q_1}M=M^{Q_2}$, the set of fixed points of $M$ under the left action of $Q_1$ (respectively the right action of $Q_2$), induces a Morita equivalence between $b_1^{Q_1}$ and $b_2^{Q_1}$. 
\item If, in addition, ${}^{Q_1}M=M^{Q_2}$ has trivial source, when considered as an $\cO((H_1/Q_1)\times(H_2/Q_2))$-module, then $M$ has trivial source, when considered as an $\cO(H_1\times H_2)$-module.
\end{enumerate}
\end{prop}

\begin{proof}
This is proved in~\cite[Propostions 4.3,4.4]{eali19}, with the added assumption that $H_1=H_2$, $Q_1=Q_2$ and $b_1=b_2$. However, the proof in this more general setting is identical.
\end{proof}

\section{Definition of $B_\varphi$}\label{sec:def}

In this section we provide a number of necessary preliminary results before defining the blocks $B_\varphi$ that will ultimately form the family of blocks with arbitrarily large $\cO$-Morita Frobenius numbers in the proof of Theorem~\ref{thm:main}.
\newline
\newline
Until further notice we fix a prime $p\neq l$ such that $p-1$ is not a power of $l$. We set $a:=v_l(p-1)$, the largest power of $l$ dividing $p-1$. For $t\in\mathbb{N}$, we define $\Omega_t$ to be the direct product of $p$ copies of $C_{l^t}$ indexed by the elements of $\mathbb{F}_p$,
\begin{align*}
\Omega_t:=\prod_{x\in\mathbb{F}_p}C_{l^t}.
\end{align*}
We also define following the subgroups of $\Omega_t$,
\begin{align*}
D_t:=\left\{(g_x)_{x\in\mathbb{F}_p}\in\Omega_t\Big{|}\prod_{x\in\mathbb{F}_p}g_x=1\right\},&&\Lambda_t:=\{(g,\dots,g)|g\in C_{l^t}\}.
\end{align*}
We set $F:=\mathbb{F}_p\rtimes\mathbb{F}_p^\times$, with multiplication given by
\begin{align*}
(x,\alpha).(y,\beta)=(x+\alpha y,\alpha\beta),
\end{align*}
for $x,y\in \mathbb{F}_p$ and $\alpha,\beta\in\mathbb{F}_p^\times$. Note that $F$ acts on $\mathbb{F}_p$ via
\begin{align*}
(x,\alpha).y=x+\alpha y,
\end{align*}
for all $x,y\in \mathbb{F}_p$ and $\alpha\in\mathbb{F}_p^\times$. Therefore, $F$ acts on $\Omega_t$ by permuting indices and $\Omega_t=D_t\times\Lambda_t$ is an $F$-stable direct decomposition of $\Omega_t$.

\begin{lem}\label{lem:FpStable}
Let $t\in\mathbb{N}$ and $\theta\in\Irr(\Omega_t)$. Viewing $\mathbb{F}_p\leq F$, $\theta$ is $\mathbb{F}_p$-stable if and only if $\theta=1_{D_t}\otimes \theta_{\Lambda_t}$, for some $\theta_{\Lambda_t}\in\Irr(\Lambda_t)$. In particular, the only irreducible, $\mathbb{F}_p$-stable character of $D_t$ is $1_{D_t}$.
\end{lem}

\begin{proof}
Certainly $\theta=1_{D_t}\otimes \theta_{\Lambda_t}$ is $\mathbb{F}_p$-stable. For the converse, suppose $\theta\in\Irr(\Omega_t)$ is $\mathbb{F}_p$-stable and $(g_x)_{x\in\mathbb{F}_p}\in D_t$. Then
\begin{align*}
\theta(g_0,g_1,\dots)&=\theta(g_0,1,1,\dots)\theta(1,g_1,1,\dots)\dots\theta(1,\dots,1, g_{p-1})\\
&=\theta\left(\prod_{x\in\mathbb{F}_p}g_x,1,1,\dots,1\right)=1.
\end{align*}
In other words, $D_t$ is contained in the kernel of $\theta$. The claim follows.
\end{proof}

From now on we fix a generator $\lambda$ of $\mathbb{F}_p^\times$. We define the group $E$ to be
\begin{align*}
\mathbb{F}_p\times\mathbb{F}_p\times\mathbb{F}_p^\times\times\mathbb{F}_p^\times\times\mathbb{F}_p^\times
\end{align*}
as a set, with multiplication given by
\begin{align*}
&(x_1,y_1,\lambda^{m_1},\lambda^{n_1},\mu_1).(x_2,y_2,\lambda^{m_2},\lambda^{n_2},\mu_2)\\
=&(x_1+\lambda^{m_1}x_2,y_1+\lambda^{n_1}y_2,\lambda^{m_1+m_2},\lambda^{n_1+n_2},\mu_1\mu_2\lambda^{n_1m_2}),
\end{align*}
for $x_1,y_1,x_2,y_2\in\mathbb{F}_p$, $\mu_1,\mu_2\in\mathbb{F}_p^\times$ and $m_1,n_1,m_2,n_2\in\mathbb{N}_0$.
Setting $Z:=\mathbb{F}_p^\times$, we have the short exact sequence
\begin{align}\label{algn:ses}
1\longrightarrow Z\xrightarrow{\eta} E\xrightarrow{\phi} F_1\times F_2\longrightarrow 1,
\end{align}
where $F_1\cong F_2\cong F$,
\begin{align*}
\eta(\mu)=(0,0,1,1,\mu)&&\text{ and }&&
\phi(x,y,\lambda^m,\lambda^n,\mu)=((x,\lambda^m),(y,\lambda^n)),
\end{align*}
for all $x,y\in\mathbb{F}_p$, $\mu\in\mathbb{F}_p^\times$ and $m,n\in\mathbb{N}_0$. We identify $Z$ with its image under $\eta$. Note that
\begin{align}\label{algn:comm}
[\widetilde{(x,\lambda^m)},\widetilde{(y,\lambda^n)}]=\lambda^{-mn}\in Z,
\end{align}
where $(x,\lambda^m)\in F_1$, $(y,\lambda^n)\in F_2$ and the tildes denote lifts to $E$. Let $\xi$ be a generator of $\mathbb{F}_p^\times$. We set
\begin{align*}
P_1&:=\{(x,0,1,1,1)\in E|x\in\mathbb{F}_p\},\text{ }P_2:=\{(0,y,1,1,1)\in E|y\in\mathbb{F}_p\},\\
P&:=P_1\times P_2\in\Syl_p(E),\\
E_{l'}&:=O_{l'}(E)=\{(x,y,\xi^m,\xi^n,\xi^r)\in E|m,n,r\in l^a\mathbb{N}_0\},\\
Z_{l'}&:=O_{l'}(Z)=E_{l'}\cap Z.
\end{align*}
We set $F_{i,l'}:=O_{l'}(F_i)$, for $i=1,2$, so $F_{1,l'}\times F_{2,l'}$ is the image of $E_{l'}$ under $\phi$. Note that, since $p-1$ is not a power of $l$, $\phi(P_i)\lneq F_{i,l'}$, for $i=1,2$.
\newline
\newline
Until further notice we fix $t_1,t_2\in\mathbb{N}$. Since $F$ acts on $D_t$, for any $t\in\mathbb{N}$, we have a natural action of $F_1\times F_2$ on $D_{t_1}\times D_{t_2}$.

\begin{lem}\label{lem:norm}
$F_1\times F_2$ acts faithfully on $D_{t_1}\times D_{t_2}$ and
\begin{align*}
N_{\Aut(D_{t_1}\times D_{t_2})}(F_{1,l'}\times F_{2,l'})=
\begin{cases}
((C_1\rtimes F_1)\times (C_2\rtimes F_2))\rtimes\langle s\rangle&\text{ if }t_1=t_2,\\
(C_1\rtimes F_1)\times (C_2\rtimes F_2)&\text{ otherwise},\\
\end{cases}
\end{align*}
where $s\in \Aut(D_{t_1}\times D_{t_2})$ is defined via $s(g,h)=(h,g)$, for all $(g,h)\in D_{t_1}\times D_{t_2}$ and $C_i:=C_{\Aut(D_{t_i})}(F_{i,l'})$, for $i=1,2$.
\end{lem}

\begin{proof}
To show $F_1\times F_2$ acts faithfully on $D_{t_1}\times D_{t_2}$ we need only show that $F$ acts faithfully on $D_t$, for any $t\in \mathbb{N}$. However, $\Lambda_t=C_{\Omega_t}(F)$ and $\Omega_t=D_t\times \Lambda_t$ so it suffices to show that $F$ acts faithfully on $\Omega_t$. This follows since the action of $F$ on $\mathbb{F}_p$ is faithful.
\newline
\newline
Again, consider the action of $F$ on $\Omega_t$ for some $t\in\mathbb{N}$. Since $\Lambda_t=C_{\Omega_t}(\mathbb{F}_p)$, $C_{D_t}(\mathbb{F}_p)$ is trivial and so we can determine $D_{t_1}$ and $D_{t_2}$ as the unique non-trivial subgroups of $D_{t_1}\times D_{t_2}$ occurring as the centraliser of some non-trival $p$-element of $F_{1,l'}\times F_{2,l'}$. Therefore, any element of $N_{\Aut(D_{t_1}\times D_{t_2})}(F_{1,l'}\times F_{2,l'})$ must respect the decomposition $D_{t_1}\times D_{t_2}$. If $t_1=t_2$, then $s$ swaps $D_{t_1}$ and $D_{t_2}$ and if $t_1\neq t_2$, then every element of $N_{\Aut(D_{t_1}\times D_{t_2})}(F_{1,l'}\times F_{2,l'})$ must leave each $D_{t_i}$ invariant, for $i=1,2$.
\newline
\newline
To complete the claim we need to prove that $N_{\Aut(D_{t_i})}(F_{i,l'})=C_i\rtimes F_i$, for $i=1,2$. However, this follows from the claim that $\Aut(F_{l'})\cong F$, where $F_{l'}:=O_{l'}(F)\cong F_{i,l'}$. We first show that $F$ acts faithfully on $F_{l'}$ via conjugation. Suppose that $g\in C_F(F_{l'})$, then $g\in C_F(\mathbb{F}_p)=\mathbb{F}_p$. As noted before the lemma, $\mathbb{F}_p$ is a proper subgroup of $F_{l'}$ and so $g\in C_{\mathbb{F}_p}(h)=\{1\}$, where $h\in F_{l'}\backslash \mathbb{F}_p$. Therefore, we have proved that $F\leq \Aut(F_{l'})$. Now let $\zeta\in\Aut(F_{l'})$. Certainly $\Aut(\mathbb{F}_p)\cong\mathbb{F}_p^\times$ and so to prove that $\zeta$ is induced by some element of $F$ we may assume that $\zeta$ fixes $\mathbb{F}_p$ pointwise. By the Schur-Zassenhaus theorem, we may, in addition, assume that $\zeta$ leaves $(\mathbb{F}_p^\times)_{l'}:=O_{l'}(\mathbb{F}_p^\times)$ invariant. Finally, since it fixes $\mathbb{F}_p$ pointwise, it must also fix $(\mathbb{F}_p^\times)_{l'}$ pointwise.
\end{proof}

\begin{defi}
We define $\tilde{G}:=(D_{t_1}\times D_{t_2})\rtimes E$, where the action of $E$ on $D_{t_1}\times D_{t_2}$ is given via $\phi$. In addition we set $G:=(D_{t_1}\times D_{t_2})\rtimes E_{l'}\leq \tilde{G}$ and for each $\varphi\in\Irr(Z_{l'})$ we set $B_\varphi:=\cO G e_\varphi$.
\end{defi}

Note that, since $D:=D_{t_1}\times D_{t_2}\lhd G$, any block idempotent of $\cO G$ is supported on $C_G(D)=D\times Z_{l'}$. Therefore, $B_\varphi$ is a block of $\cO G$ with defect group $D$.

\section{Arbitrarily large $\cO$-Morita Frobenius numbers}\label{sec:main}

Adopting the notation of Lemma~\ref{lem:norm} we have the following.

\begin{lem}\label{lem:autos}
Let $\varphi\in\Irr(Z_{l'})$.
\begin{enumerate}
\item For each $\zeta\in (C_1\rtimes F_1)\times (C_2\rtimes F_2)$, there exists $\delta\in\Aut(G)$ such that $\delta\downarrow_D=\zeta$ and $\delta\downarrow_{Z_{l'}}=\Id_{Z_{l'}}$. In particular, $\delta$ induces an $\cO$-algebra automorphism of $B_\varphi$.
\item If $t_1=t_2$, there exists $\delta\in\Aut(G)$ such that $\delta\downarrow_D=s$ and $\delta\downarrow_{Z_{l'}}$ is given by inversion. In particular, $\delta$ induces an $\cO$-algebra homomorphism $B_\varphi\to B_{\varphi^{-1}}$.
\end{enumerate}
\end{lem}

\begin{proof}$ $
\begin{enumerate}
\item For each $\zeta\in C_1\times C_2$, we define $\delta\in\Aut(G)$ via $\delta\downarrow_D=\zeta$ and $\delta\downarrow_{E_{l'}}=\Id_{E_{l'}}$. For each $\zeta\in F_1\times F_2$, we define $\delta\in\Aut(G)$ to be given by conjugating by some $g\in E\leq \tilde{G}$, a lift of $\zeta$.
\item We define $\delta\in\Aut(G)$ via $\delta\downarrow_D=s$ and
\begin{align*}
\delta\downarrow_{E_{l'}}(x,y,\lambda^m,\lambda^n,\lambda^r)=(y,x,\lambda^n,\lambda^m,\lambda^{mn-r}),
\end{align*}
for all $x,y\in\mathbb{F}_p$ and $m,n,r\in l^a\mathbb{N}_0$, on $E_{l'}$.
\end{enumerate}
\end{proof}

\begin{lem}\label{lem:Dkernel}
Any $\chi\in\Irr(G)$ reduces to an irreducible Brauer character if and only if $\chi\in\Irr(G,1_D)$.
\end{lem}

\begin{proof}
Since $D\lhd G$ and $G/D\cong E_{l'}$ is an $l'$-group, $D$ is contained in the kernel of every simple $kG$-module and every irreducible Brauer character is determined by its restriction to $E_{l'}$. Therefore, we have a bijection between $\IBr(G)$ and $\Irr(E_{l'})$ given by restriction to $E_{l'}$ and through this bijection we can identify the decomposition map
\begin{align*}
\mathbb{Z}\Irr(G)\to\mathbb{Z}\IBr(G)
\end{align*}
with the restriction map
\begin{align*}
\mathbb{Z}\Irr(G)\to\mathbb{Z}\Irr(E_{l'}).
\end{align*}
It therefore remains to show that for any $\chi\in\Irr(G)$, $\chi\downarrow_{E_{l'}}$ is irreducible if and only if $\chi\in\Irr(G,1_D)$.
\newline
\newline
If $\chi\in\Irr(G,1_D)$, then certainly $\chi\downarrow_{E_{l'}}$ is irreducible.
\newline
\newline
For the converse let $1_D\neq\theta\in\Irr(D)$. By Lemma~\ref{lem:FpStable}, $\Stab_P(\theta)=\{1\}$, $P_1$ or $P_2$. Therefore, any $\chi\in\Irr(G,\theta)$ must satisfy the following condition. Either $\chi\downarrow_{P_1}$ has trivial and non-trivial, irreducible constituents or $\chi\downarrow_{P_2}$ has trivial and non-trivial, irreducible constituents. However, by considering orbits of $\Irr(P)$ under the action of $E_{l'}$, we get that for any $\xi\in\Irr(E_{l'})$, $\xi\downarrow_{P_1}$ does not have both trivial and non-trivial, irreducible constituents and $\xi\downarrow_{P_2}$ does not have both trivial and non-trivial, irreducible constituents. Therefore, $\chi\downarrow_{E_{l'}}$ cannot be irreducible.
\end{proof}

\begin{prop}\label{prop:ME}
Let $\varphi,\vartheta\in\Irr(Z_{l'})$. Then $B_\varphi$ is Morita equivalent to $B_\vartheta$ if and only if $\varphi=\vartheta$ or $t_1=t_2$ and $\varphi=\vartheta^{-1}$.
\end{prop}

\begin{proof}
Certainly if $\varphi=\vartheta$, then $B_\varphi$ is Morita equivalent to $B_\vartheta$. Also, if $t_1=t_2$ and $\varphi=\vartheta^{-1}$, then $B_\varphi$ and $B_\vartheta$ are isomorphic via part (2) of Lemma~\ref{lem:autos}.
\newline
\newline
Conversely, suppose $M$ is an $B_\varphi$-$B_\vartheta$-bimodule inducing a Morita equivalence between $B_\varphi$ and $B_\vartheta$ and $\sigma:\Irr(B_\varphi)\to\Irr(B_\vartheta)$ the corresponding bijection of characters. Then, by Lemma~\ref{lem:Dkernel}, $\sigma$ restricts to a bijection $\Irr(B_\varphi,1_D)\to\Irr(B_\vartheta,1_D)$ and by part (1) of Proposition~\ref{prop:indMor}, ${}^DM=M^D$ induces a Morita equivalence between $B_\varphi^D$ and $B_\vartheta^D$. However, $G/D$ is an $l'$-group and so ${}^DM=M^D$ certainly has trivial source. Therefore, by part (2) of Proposition~\ref{prop:indMor}, $M$ must also have trivial source.
\newline
\newline
It now follows from~\cite[7.6]{pu99} that $M$ is a direct summand of $\cO_{\Delta\gamma}\uparrow^{G\times G}$, for some $\gamma\in\Aut(D)$, where $\Delta\gamma:=\{(d,\gamma(d))|d\in D\}$ and $\cO_{\Delta\gamma}$ denotes the trivial $\cO(\Delta\gamma)$-module. Therefore, $M$ is a direct summand of 
\begin{align*}
e_\varphi(\cO_{\Delta\gamma}\uparrow^{(D\times Z_{l'})\times (D\times Z_{l'})}\uparrow^{G\times G})e_\varphi\cong& (e_\varphi(\cO_{\Delta\gamma}\uparrow^{(D\times Z_{l'})\times (D\times Z_{l'})})e_\varphi)\uparrow^{G\times G}\\
\cong&({}_\gamma(\cO D)\otimes_\cO {}_\varphi\cO_\vartheta)\uparrow^{G\times G},
\end{align*}
where ${}_\gamma(\cO D)$ denotes the $\cO D$-$\cO D$-bimodule $\cO D$, with the canonical right action of $\cO D$ and the left action of $\cO D$ given via $\gamma$ and ${}_\varphi\cO_\vartheta$ denotes the $\cO Z_{l'}e_\varphi$-$\cO Z_{l'}e_\vartheta$-bimodule $\cO$, with the canonical left $\cO Z_{l'}e_\varphi$ and right $\cO Z_{l'}e_\vartheta$ actions.
\newline
\newline
We now analyse $S:=\Stab_{G\times G}({}_\gamma(\cO D)\otimes_\cO {}_\varphi\cO_\vartheta)$. First note that $M\cong N\uparrow^{G\times G}$, for some indecomposable $\cO S$-module $N$ such that ${}_\gamma(\cO D)\otimes_\cO {}_\varphi\cO_\vartheta$ is a direct summand of $N\downarrow_{(D\times Z_{l'})\times (D\times Z_{l'})}$. Therefore,
\begin{align}\label{algn:ineq}
\begin{split}
&|D|=\rk_{\cO}({}_\gamma(\cO D)\otimes_\cO {}_\varphi\cO_\vartheta)\leq\rk_{\cO}(N),\\
&\rk_{\cO}(N)[G\times G:S]=\rk_{\cO}(M)\leq\rk_{\cO}(B_\varphi)=[G:Z_{l'}],
\end{split}
\end{align}
where the inequality on the second line follows from Lemma~\ref{lem:rank}. In particular, $[G\times G:S]\leq [G:D\times Z_{l'}]$, which rearranges to
\begin{align}\label{algn:bdd_below}
[S:(D\times Z_{l'})\times(D\times Z_{l'})]\geq [G:D\times Z_{l'}]
\end{align}
with equality if and only if we have equality throughout (\ref{algn:ineq}). Next let $(g,h)\in S$. Since ${}_\gamma(\cO D)$ is $S$-stable and $\Delta\gamma$ is the unique vertex of ${}_\gamma(\cO D)$, we have
\begin{align*}
(g,h)(\Delta\gamma)(g,h)^{-1}=\Delta\gamma.
\end{align*}
In other words,
\begin{align}\label{algn:normalise}
gdg^{-1}=\gamma^{-1}(h\gamma(d)h^{-1}),
\end{align}
for all $d\in D$. Now, by Lemma~\ref{lem:norm}, $F_{1,l'}\times F_{2,l'}\cong G/(D\times Z_{l'})$ acts faithfully on $D$ and so $g(D\times Z_{l'})$ determines $h(D\times Z_{l'})$. In particular,
\begin{align*}
[S:(D\times Z_{l'})\times(D\times Z_{l'})]\leq [G:D\times Z_{l'}].
\end{align*}
Together with (\ref{algn:bdd_below}) and the subsequent sentence, this tells us that we have equality throughout (\ref{algn:ineq}) and (\ref{algn:bdd_below}). Equality in (\ref{algn:bdd_below}) and the sentence following (\ref{algn:normalise}) give that there exists $\zeta\in\Aut(G/(D\times Z_{l'}))$ such that
\begin{align*}
S=\{(g,h)\in G\times G|\zeta(g(D\times Z_{l'}))=h(D\times Z_{l'})\}.
\end{align*}
Furthermore, equality throughout (\ref{algn:ineq}) tells us that ${}_\gamma(\cO D)\otimes_\cO {}_\varphi\cO_\vartheta$ extends to an $\cO S$-module. Now (\ref{algn:normalise}) implies that $\gamma\in N_{\Aut(D)}(F_{1,l'}\times F_{2,l'})$ and $\zeta$ is the corresponding automorphism of $F_{1,l'}\times F_{2,l'}$, once we have identified $G/(D\times Z_{l'})$ with $F_{1,l'}\times F_{2,l'}\leq\Aut(D)$.
\newline
\newline
By Lemmas~\ref{lem:norm} and~\ref{lem:autos}, we may assume that $\gamma=\Id_D$. In particular, $\zeta=\Id_{F_{1,l'}\times F_{2,l'}}$ and so $\cO D\otimes_\cO {}_\varphi\cO_\vartheta$ extends to a module for
\begin{align*}
\cO(((D\times Z_{l'})\times (D\times Z_{l'})).(\Delta E_{l'})).
\end{align*}
So $kD\otimes_k{}_\varphi k_\vartheta$ extends to a module for
\begin{align*}
k(((D\times Z_{l'})\times (D\times Z_{l'})).(\Delta E_{l'})),
\end{align*}
where
\begin{align*}
kD\otimes_k{}_\varphi k_\vartheta:=k\otimes_{\cO}(\cO D\otimes_\cO {}_\varphi\cO_\vartheta),
\end{align*}
which we identify with $kD$ in the obvious way. As a $k((D\times Z_{l'})\times(D\times Z_{l'}))$-module, the radical of $kD$ is $J(kD)$, the Jacobson radical of $kD$ as a ring. We now study the $1$-dimensional $k(\Delta E_{l'})$-module $kD/J(kD)$.
\newline
\newline
Viewing $\lambda\in Z$, (\ref{algn:comm}) gives that $\lambda^{l^{-2a}}\in[E_{l'},E_{l'}]$ but $\lambda^{l^{-2a}}$ generates $Z_{l'}$ so $Z_{l'}\leq[E_{l'},E_{l'}]$ and $(kD/J(kD))\downarrow_{\Delta Z_{l'}}$ must be the trivial $k(\Delta Z_{l'})$-module. However, $(kD/J(kD))\downarrow_{\Delta Z_{l'}}$ is also the $1$-dimensional $k(\Delta Z_{l'})$-module corresponding to $\Delta Z_{l'}\to k^\times$, $(z,z)\mapsto\overline{\varphi.\vartheta^{-1}(z)}$. Since $Z_{l'}$ is an $l'$-group, this implies $\varphi=\vartheta$, as required.
\end{proof}

\begin{thm}\label{thm:main}
For every prime $l$ and $n\in\mathbb{N}$, there exists an $\cO$-block $b$ with $\omf(b)=n$.
\end{thm}

\begin{proof}
Let $p$ be a prime different from $l$ such that $p\equiv 1 \mod (l^n-1)$, the existence of which is guaranteed by the Dirichlet prime number theorem. In particular, $p-1$ is not a power of $l$ and so we can adopt all the notation from this and the previous section.
\newline
\newline
Let $t_1\neq t_2\in\mathbb{N}$ and $\varphi$ a faithful character of $Z_{l'}$, in particular, $\varphi$ has order $(p-1)/l^a$ which is divisible by $l^n-1$. Now set $\vartheta:=\varphi^{(p-1)/(l^a(l^n-1))}$ so $\vartheta$ has order $l^n-1$. One can quickly verify that $\overline{e_\vartheta}^{(l^m)}=\overline{e_{\vartheta^{l^m}}}$ and hence that $B_\vartheta^{(l^m)}=B_{\vartheta^{l^m}}$, for all $m\in\mathbb{N}$. Proposition~\ref{prop:ME} now implies that $\omf(B_\vartheta)$ is the smallest $m\in\mathbb{N}$ such that $\vartheta^{l^m}=\vartheta$. Therefore, $\omf(B_\vartheta)=n$.
\end{proof}

\begin{ack*}
The author would like to express gratitude to Prof. Burkhard K\"ulshammer for hosting the author at Friedrich-Schiller-Universit\"at Jena, where much of this article was written.
\end{ack*}

\end{document}